\documentclass{amsart}

\usepackage{amsmath, amssymb}
\usepackage{amscd}
\usepackage{amsfonts}
\usepackage{latexsym}

\newtheorem{theorem}{Theorem}

\newtheorem{lemma}{Lemma}[section]

\newtheorem{definition}{Definition}

\def\l{\left}
\def\r{\right}

\newcommand \ad{\text{ad}\thinspace}

\newcommand\la{\langle}
\newcommand\ra{\rangle}
\renewcommand{\frak}{\mathfrak}

\begin{document}

\title {Borcherds' proof of the Conway-Norton conjecture}
\author{Elizabeth Jurisich }

\maketitle

\begin{center} Department of Mathematics\\
  The College of Charleston\\
   Charleston SC 29424\\ jurisiche@cofc.edu \end{center}

\begin{abstract} We give a summary of R. Borcherds' solution (with  
some modifications) to the following part of the Conway-Norton  
conjectures:
Given the Monster $\mathbb M$ and Frenkel-Lepowsky-Meurman's  
moonshine module $V^\natural$, prove the equality between the graded  
characters of the elements of $\mathbb M$ acting on $V^\natural$  
(i.e., the McKay-Thompson
series for $V^\natural$) and the modular functions provided by Conway  
and Norton. The equality is established using the homology of a  
certain subalgebra of the monster Lie algebra, and the
Euler-Poincar\'e identity.
\end{abstract}

\section{Introduction}
In this paper we present a summary of R. Borcherds' proof of part of  
the Conway-Norton ``monstrous moonshine" conjectures:  the proof that  
the McKay-Thompson series of the Monster simple group acting on the  
known structure $V^\natural$ \cite{FLM} do indeed correspond to the  
Hauptmoduls presented in Conway and Norton
\cite{CN}. Interested readers should certainly consult primary  
sources, a few of which are \cite{MR843307}, \cite{FLM1}, \cite{FLM},  
and \cite{B3}. The simplification of the original proof, presented in  
this paper, can be found in \cite{Jur2} and \cite{jlw}. See also  
Borcherds' survey articles about moonshine \cite{MR1341831} and \cite 
{MR1660657}. A brief overview of the historical development of the  
subject can be found in \cite{FLM}.  What the reader will find in  
this paper is an outline of the proof itself, with references to  
particular results needed to establish the equality between the 
McKay-Thompson series for $V^\natural$ and the Laurent expansions of the  
relevant modular functions.

Given a group $G$ and a $\mathbb Z$-graded $G$-module $V = \coprod_ 
{ n \in \mathbb Z}V_{[n]}$, with $\text{dim} V_{[n]} <\infty$ for all  
$n\in \mathbb Z$ truncated so $V_{[k]} = 0$ for $k <N$ for some fixed  
$N \in \mathbb Z$, the McKay-Thompson series for $g \in G$ 
acting on $V$ is defined to be the graded trace
$$T_g (q) = \sum_{n>N} \text{tr}(g|V_{[n]}) q^n.$$
The moonshine conjectures of Conway and Norton in \cite{CN} include  
the conjecture that there
should be
an infinite-dimensional representation of the (not yet constructed)  
Fischer-Griess Monster simple group $\mathbb M$
such that the McKay-Thompson series $T_g$ for $g \in \mathbb M$  
acting on $V$
have coefficients that are equal to the coefficients of the $q$-series  
expansions of certain modular functions. After the construction of $ 
\mathbb M$ \cite{Gr}, a ``moonshine module" $V^\natural$ for the  
Monster simple group was
constructed \cite{FLM1}, \cite{FLM} and many of its properties,  
including the
determination of some of its McKay-Thompson series, were proven.
The critical example involves the modular function $j(\tau)$,  a  
generator for the field of functions invariant under the action of  
$SL_2(\mathbb Z)$ on the upper half plane.  Let $q = e^{2\pi i \tau} 
$.  Then the normalized $q$-series expansion denoted by $J(q) = j(q)-  
744$ is one of the Hauptmoduls that occur in the moonshine  
correspondence \cite{CN}. Results of \cite{FLM} include the vertex  
operator algebra structure of $V^\natural$, with the conformal grading  
$V^\natural = \coprod_{i\geq 0} V^\natural_i$, as an $\mathbb
M$-module, and the graded dimension correspondence $\text{dim}\,V^ 
\natural_i \leftrightarrow c(i-1)$ to the coefficients of  the  
modular function $J(q) = \sum_{n\geq -1} c(n) q^n$. In other words,  
shifting the grading by defining $  V^\natural_{[i-1]}= V^\natural_i 
$, we have $T_e(q) =  J(q)$ as formal series, where $e\in \mathbb M$ is  
the identity element.

After the above results, the nontrivial problem of computing the rest of
the McKay-Thompson series of Monster group elements acting
on $V^\natural$ remained. Borcherds  showed in \cite{B3}
that the McKay-Thompson series are the expected modular functions.
The argument can be summarized as follows:  Borcherds establishes a  
product formula
$$
p(J(p)-J(q))=  \prod_{ { i= 1,2, \ldots},  { j = -1,1,\ldots}}
(1-p^iq^j)^{c(ij)}.
$$
This formula is used in the proof, but first note that the formula  
leads to recursion formulas for the coefficients of the $q$-series  
expansion of $J(q)$, and hence to recursions for the dimensions of  
the homogeneous components of $V^\natural$.   The approach of \cite 
{B3} is to establish a product formula involving all of the
McKay-Thompson series $T_g(q)$ of elements of $g\in \mathbb M$ acting on $V^ 
\natural$, analogous to the above identity for $T_e(q)= J(q)$. The  
more general product formula in turn leads to a set of recursion  
formulas that determine the coefficients of the series, given a  
(large) set of initial data. For example, to determine $\sum {\rm tr} 
(g|V^\natural_{i+1})q^{i } = \sum_{i\geq -1} c_g(i )q^i$ it is  
sufficient to compute four of the first five coefficients of $\sum   
\text{tr}(g^k|V^\natural_{i+1})q^{i }$, $k \in \mathbb Z$, $g\in  
\mathbb M$.  One also determines that the Hauptmoduls listed in \cite 
{CN} satisfy the same recursion relations and initial data as the  
McKay-Thompson series for $V^\natural$.

The crucial product identity for the McKay-Thompson series is  
obtained by Borcherds from the Euler-Poincar\'e identity for the  
homology groups of a particular Lie algebra, the ``monster" Lie  
algebra. This Lie algebra is constructed using the tensor product of  
the vertex operator algebra $V^\natural$ and a vertex algebra  
associated with a two-dimensional Lorentzian lattice.
The  monster Lie algebra , $\mathfrak m$, constructed in \cite{B3}  
is  an infinite-dimensional $\mathbb Z \times \mathbb Z$-graded Lie  
algebra.  The Lie algebra $\frak m$ is then shown to be a generalized  
Kac-Moody algebra, or a Borcherds algebra.  The ``No-ghost" theorem of  
string theory is used as a step in establishing an isomorphism  
between the
$\mathbb Z \times \mathbb Z$-homogeneous components of $\mathfrak m$  
and the weight spaces of $V^\natural$. The product formula for $p(J(p)- 
J(q))$ is interpreted as the denominator formula for the Lie algebra $ 
\mathfrak m$ and used to determine the simple roots. The results  
pertaining to the homology groups of Lie algebras of \cite{GL} are  
then extended to include the class of Borcherds algebras and  
applied to a subalgebra $\mathfrak n^-$ of
$\mathfrak m$ to obtain the desired family of identities. Of course,  
the recursions and initial data must also be established for the  
Hauptmoduls; see \cite{Ko}  and  \cite{MR1371745}.

In this paper we will actually discuss a modification of Borcherds  
proof, in which it is not necessary to generalize the homology  
results of \cite{GL}. We shall compute the homology groups as in   
\cite{Jur2}, \cite{jlw}  with  respect to a smaller subalgebra  $ 
\mathfrak u^-\subset \frak n^-$. We shall use $\mathfrak u^-$ because  
it is a free Lie algebra, and therefore computing the homology groups  
is straightforward.  The Euler-Poincar\'e identity applied to the  
subalgebra $\mathfrak u^-$ and the trivial $\mathfrak u^-$-module
$\mathbb C$ leads to recursions sufficient to establish the  
correspondence between the McKay-Thompson series for $V^\natural$
and the Hauptmoduls specified by Conway and Norton \cite{CN}.

\section{Vertex operator algebras.}

We begin by recalling the definition of vertex operator algebra and  
vertex algebra. The following definition is a variant of Borcherds'  
original definition in
\cite{MR843307}. For a detailed discussion the reader can consult  
\cite{FLM}, \cite{FHL}, \cite{DL}.

\begin{definition} A {\it vertex operator algebra}, $(V,Y,{\bf
1},\omega)$, consists of a vector space $V$, distinguished vectors  
called
the {\it vacuum vector} $\bf 1$ and the {\em conformal vector}
$\omega$, and a linear map $Y(\cdot, z): V
\rightarrow (\mbox{End }V)[[z,z^{-1}]]$ which is a generating function
for operators $v_n$, i.e., for $v \in V,\ Y(v,z)= \sum_{n \in \mathbb Z}
v_n z^{-n-1}$, satisfying the following conditions:
\begin{description}
\item[(V1)] $V =\coprod_{n \in \mathbb Z} V_{n}$; for $v \in V_{n}$,  
$n= {\rm wt}
(v)$
\item[(V2)] $\dim V_{n} < \infty$ for $n \in \mathbb Z$
\item[(V3)] $V_{n}=0 $ for $n$ sufficiently small
\item[(V4)] If $u,v \in V$ then $u_nv =0$ for $n$ sufficiently large
\item[(V5)] $Y({\bf 1},z) = 1$
\item[(V6)] $Y(v,z){\bf 1} \in V[[z]]$ and $\lim_{z \rightarrow 0}
Y(v,z){\bf 1}= v$, i.e., the {\it creation property} holds
\item[(V7)] The following {\it Jacobi identity} holds:
$$ z_0^{-1} \delta \l( {\frac {z_1-z_2} {z_0}}\r)Y(u, z_1)Y(v,z_2)
  -z_0^{-1}\delta \l( {\frac {z_2-z_1} {-z_0}}\r)Y(v,z_2)Y(u,z_1) $$
\begin{equation}
  = z_2^{-1} \delta \l({\frac{z_1-z_0} { z_2}}\r) Y(Y(u,z_0)v,z_2).
\end{equation}
\end{description}
The following conditions relating to the vector $\omega$ also hold;
\begin{description}
\item[(V8)] The operators $\omega_n$ generate a Virasoro algebra  
i.e., if
we let $L(n) = \omega_{n+1}$ for $n \in \mathbb Z$ then
\begin{equation}
[L(m),L(n)] = (m-n)L(m+n) + (1/12) (m^3-m)\delta_{m+n,0}(\mbox{rank} 
\, V)
\end{equation}
\item[(V9)] If $v \in V_{n}$ then $L(0)v = ({\rm wt}\, v)v = nv$
\item[(V10)] ${\frac d {dz} }Y(v,z)= Y(L(-1)v,z)$.
\end{description}\end{definition}

\begin{definition}  A {\it vertex algebra} (with conformal vector)
  $(V,Y,{\bf 1}, \omega)$ is a vector space $V$ with all of the above  
properties except for
$\bf V2$ and $\bf V3$.
\end{definition}

For a vertex algebra $V$, and  $v \in V$ with $\text{wt} \, v = n$,  
let $(-z^{-2})^{L(0)} v = (-z^{-2})^n v$. This action extends  
linearly to all of $V$.

\begin{definition}
A bilinear form on a vertex algebra $V$ is {\em invariant} if for  
$v,u,w \in V$
$$
( Y(v,z)u,w  ) =
( u, Y(e^{zL(1)}(-z^{-2})^{L(0)}v,z^{-1})w  ) .
$$
\label{def:inv}
\end{definition}

We note that an invariant form satisfies $( u, v )=0$ unless ${\rm  
wt} (u) ={\rm wt} (v)$ for $u,v$ homogeneous elements of $V$.

The tensor
product of vertex algebras is also a vertex
algebra \cite{FHL}, \cite{DL}. Given two vertex algebras $(V,Y,{\bf 1} 
_V, \omega_V)$ and
$(W,Y,{\bf 1}_W, \omega_W)$ the vacuum of
$V \otimes W$ is ${\bf 1}_V \otimes {\bf 1}_W$
and the conformal vector $\omega$ is given by $\omega_V \otimes
{\bf 1}_W + {\bf 1}_V \otimes \omega_W$.
If the vertex algebras $V$
and $W$ both have invariant forms in the sense of Definition \ref 
{def:inv} then it follows from the definition of the tensor product
that the form on $V \otimes W$ given by the product of the forms on
$V$ and $W$ is also invariant.

One large and important class of vertex algebras are those associated  
with even lattices.
Although the moonshine module $V^\natural$ is not a vertex operator  
algebra associated with a lattice (it is a far more complicated  
object), it is constructed using the vertex operator algebra  
associated to the Leech lattice. The vertex algebra used in the proof  
of the moonshine correspondence is the tensor product of the  
moonshine module and a vertex algebra associated with a two- 
dimensional Lorentzian lattice.

Given an even lattice $L$ the vertex algebra $V_L$ \cite{MR843307}  
associated to the lattice has
underlying vector space
$$V_L=S({\hat {\frak h}^-}_\mathbb Z)\otimes \mathbb C\{L\}.$$
(We are using the notation and constructions in \cite{FLM}.)
Here we take $\frak h= L \otimes_{\mathbb Z} \mathbb C$, and
$\hat {\frak h}^-_{\mathbb Z}$ is the
negative part of the Heisenberg algebra (with $c$ central) defined by
$$\hat {\frak h}_{\mathbb Z} =\coprod_{n \in \mathbb Z} \frak h  
\otimes t^n
\oplus {\mathbb C} c \subset \frak h \otimes {\mathbb C} [[t]] \oplus
{\mathbb C} c,$$
so that
$$\hat {\frak h}^-_\mathbb Z =
\coprod_{n <0} \frak h \otimes t^n.$$
  The symmetric algebra on $\hat {\frak h}_{\mathbb Z}^-$ is denoted  
$S(\hat
{\frak h}^-_\mathbb Z)$. Let $\hat L$ be a central extension of $L$  
by a group of order $2$,
i.e.,
$$ 1 \rightarrow \la \kappa | \kappa^2 =1\ra \rightarrow \hat {L}
  {\overset  {- }{ \rightarrow}} L \rightarrow 1,$$
with commutator map given by $\kappa^{\la \alpha, \beta\ra}$, $
\alpha, \beta \in L$.
  Define $\mathbb C \{L\}$ to be the induced
module
$\mbox{Ind}_{\la \kappa \ra}^{\hat L} {\mathbb C}$, where $\kappa$  
acts on $\mathbb C$ as multiplication by $-1$.

If $a \in {\hat L}$
denote by $\iota(a)$ the element $a \otimes 1 \in \mathbb C \{L\}$. We
will use the
notation $\alpha(n) = \alpha \otimes t^n \in S(\hat {\frak h}^-_ 
\mathbb Z)$.
The vector space $V_L$ is spanned by elements of the form:
\begin{equation}
\alpha_1(-n_1)\alpha_2(-n_2)\ldots\alpha_k(-n_k)\iota(a)\label{eq:span}
\end{equation}
where $a\in \hat L, \alpha_i \in \mathfrak h$ and $n_i \in \mathbb N$.
The space $V_L$, equipped with $Y(v,z)$ as defined in \cite{FLM}
satisfies properties $\bf V1$ and $\bf V4 - V10$, so
is a vertex algebra with conformal vector $\omega$.

A vertex algebra  $V_L$ constructed from an even lattice $L$  
automatically has an invariant bilinear form \cite{MR843307}, which  
can be defined using the contragredient module $V'_L$ \cite{FHL}.

\section{Construction of the monster Lie algebra from the moonshine  
module.}

The moonshine module $V^\natural$, a graded $\mathbb M$-module and vertex  
operator algebra, is constructed in \cite 
{FLM}.  The following results describing the structure and properties  
of $V^\natural$ appear in Corollary 12.5.4 and  Theorem 12.3.1 of  
\cite{FLM}, part of  which we restate here for the convenience of the  
reader. The invariance of the form in the sense of Definition \ref 
{def:inv} follows from the construction and results of \cite{LiHS};  
see \cite{Jur2}. Recall that $J(q)$ denotes the Laurent, or
$q$-series, expansion of the modular function $j(\tau)$, normalized so  
that the coefficient of $q^0$ is zero.
\begin{theorem}
\begin{enumerate}\item The graded dimension of the moonshine module  
$V^\natural$ is $J(q)$.
\item $V^\natural$ is a vertex operator algebra of rank $24$.
\item $\mathbb M$ acts in a natural way as automorphisms as of the vertex operator  
algebra $V^\natural$, i.e.,
$$gY(v,z)g^{-1} = Y(gv,z)$$
for $g \in \mathbb M, v \in V^\natural$
\item There is an invariant positive definite hermitian form $(\cdot,  
\cdot)$ on $V^\natural$ which is also invariant under $\mathbb M$.
\end{enumerate}
\end{theorem}

In \cite{B3} the monster Lie algebra is constructed using $V^\natural 
$ and the vertex algebra associated with a Lorentzian lattice as  
follows.  Let ${ \Pi_{1,1}}=\mathbb Z \oplus \mathbb Z$ be the rank  
two Lorentzian lattice with bilinear form
$\la \cdot,\cdot \ra$ given by the matrix
$\left(\begin{array}{cc}
       0 & -1 \\
       -1 & 0
\end{array}\right)$. The vertex algebra $V_{ \Pi_{1,1} }$ has a  
conformal vector, and is given the structure of a trivial $\mathbb
M$-module.  Since $V_{  \Pi_{1,1} }$ is a vertex algebra  associated with  
an even lattice  it has an invariant bilinear form, which we consider  
as $\mathbb M$-invariant under the trivial group action. Note that $V_ 
{  \Pi_{1,1} }$ is not a vertex operator algebra, because it does not  
satisfy conditions $\bf (V2)$ or $\bf (V3)$. For example, the weight  
of an element of the form (\ref{eq:span}) is $\sum_{i=1}^k  
n_i + {\frac 1 2}\left< a,a\right>$, $n_i > 0 \in \mathbb Z$, $a \in  
\Pi_{1,1}$, which can be less than zero and arbitrarily large in  
absolute value.

\begin{lemma} The tensor product  $V =  V^\natural\otimes V_{  \Pi_ 
{1,1} }$ is a vertex algebra with conformal vector, and an invariant  
bilinear form, which is also $\mathbb M$-invariant. \label{3.1}
\end{lemma}

Given a vertex
operator algebra $V$, or a vertex algebra $V$ with conformal vector $ 
\omega$
and therefore an action of the Virasoro algebra, let $$P_{i} = \{ v  
\in V
| L(0)v = iv , L(n)v = 0 \mbox{ if } n>0\}.$$ Thus $P_{i}$ consists of
the lowest weight vectors for the Virasoro algebra of conformal
weight $i$. ($P_{1}$ is called the {\it physical space}). Let $u\in  
P_0$, then  $\text{wt} L(-1) u = \text{wt}\omega + \text{wt} u -1 = 1 
$ and $L_{-1} P_0 \subset P_1$.

\begin{lemma}
The space $P_{1}/L(-1)P_{0}$ is a Lie algebra with bracket
given by $$[u +L(-1)P_{0} ,v+L(-1)P_{0}] = u_0v + L(-1)P_{0}.$$
For $u,v \in P_1$. \label{3.2}
\end{lemma}

  \begin{proof}
  Let $u,v \in P_1$. By formula (8.8.7) of \cite{FLM} (see \cite 
{MR843307})
\begin{equation}
  Y(u,z)v = e^{zL(-1)} Y(v,-z) u. \label{eq:antisym}
\end{equation}
Taking coefficients of $z^{-1}$ on both sides of (\ref{eq:antisym})  
yields
$$u_0v = -v_0u + \sum_{k=1}^\infty (-1)^{k-1} L(-1)^k v_k u \in -v_0u  
+ L(-1) P_0.$$
Thus the bracket is anti-symmetric.
Let $u,v,w \in P_1$ The Jacobi identity $\bf (V7)$  implies
$$ u_0v_0 w -v_0u_0 w = (u_0v)_0 w
$$
$$(u_0(v_0 w))-(v_0(u_0w))-((u_0v)_0w) = 0.$$
Since we have shown anti-symmetry (modulo $L(-1) P_0$) the above is  
equivalent to the usual Lie algebra Jacobi identity for
  $[ \cdot , \cdot ] $ on $P_1/L(-1) P_0$.
  \end{proof}

  We can now give the definition of Borcherds' monster Lie algebra.
  The tensor product  $V =  V^\natural\otimes V_{  \Pi_{1,1} }$ is a  
vertex algebra with conformal vector, and invariant bilinear form.
This form
induces a bilinear form on the Lie
algebra $P_{1}/L_{-1}P_{0}$. Note that if $u,v,w \in P_1 $, then  by  
invariance, and the fact that $L(1) u = 0$
\begin{multline} (Y(u,z) v,w) = ( v, Y(e^{zL(1)}(-z^{-2})^{L(0)}u,z^ 
{-1})w)  \\
= -(v, Y(u,z^{-1}) z^{-2} w )
  = - (v,\sum_{n \in \mathbb Z} u_n z^{n-1}).
\end{multline}
Taking the coefficient of $z^{-1}$  we have  for $u, v, w \in P_1/L 
(-1) P_0$
  $$  (u_0 v, w) = -( v  , u_0w ),$$
and so the form on $P_1/L(1) P_0$ is invariant in the usual Lie  
algebra sense.

In addition to the weight grading, the vertex algebra $V^\natural 
\otimes V_{  \Pi_{1,1} }$ is graded by the lattice $\Pi_{1,1}$. For  
$u,v$ elements of degree $r,s \in  \Pi_{1,1}$,  the invariant form  
satisfies $(u,v) = 0$ unless $r = s$.

Let $  N (\cdot, \cdot)$ denote the nullspace of the bilinear form on  
$P_1$ so
  $  N (\cdot, \cdot)=\{u \in  P_{1} \, |\,  (u,v) = 0 \quad \forall  
v \in   P_{1}  \}$. Since, for $u = L(-1) v, v \in P_0$, $w \in P_1$,  
it is immediate that
  $(L(-1) v, w ) = ( v, L(1) w ) = 0$, we see $L(-1) P_0 \subset N 
(\cdot, \cdot)$.  This in conjunction with  Lemma \ref{3.2} ensures  
the following is a Lie algebra.

\begin{definition} The {\it monster Lie algebra} $\frak m$
is defined by
$$\frak m
=P_{1}/ N(\cdot,\cdot) .$$
\end{definition}

The monster Lie algebra is graded by the Lorentzian lattice $\Pi_{1,1} 
$ by construction. Elements of $\frak m$ can be written as $\sum u
\otimes v e^r$, where $u \in V^{\natural}$ and $ ve^r = v \iota (e^r)
  \in V_{\Pi_{1,1}}$. Here, a section of the map $\hat \Pi_{1,1}
  {\overset { -} { \rightarrow}} \Pi_{1,1}$ has been chosen so that  
$e^r \in
{\hat \Pi_{1,1}}$ satisfies $\overline{e^r}= r
\in \Pi_{1,1} $. There is a grading of $\frak m$ by the lattice  
defined by
$\deg (u \otimes ve^r)=r$. It follows from the construction that the  
Lie algebra $\mathfrak m$ has a Lie invariant bilinear form, whose  
radical is zero.

In order to establish the equality between the coefficients of the  
McKay-Thompson series for $V^\natural$ and the given Hauptmoduls, it  
is necessary to determine the dimensions of the components of $ 
\mathfrak m$ of degrees $r \in  \Pi_{1,1}$.
Borcherds  \cite{B3} computes the dimensions by using Theorem 2  
below, which uses the No-ghost theorem of string theory. For a proof  
of the No-ghost theorem see \cite{GT}, \cite{B3}, or the appendix of  
\cite{Jur2} for one written more algebraically.

\begin{theorem}
Let $V$ be a vertex operator algebra
with the following properties:
\begin{itemize}
  \item [i.] $V$ has a symmetric invariant nondegenerate bilinear form.
  \item [ii.] The central element of the Virasoro algebra acts as
multiplication by 24.
  \item [iii.] The weight grading of $V$ is an $\mathbb N$-grading of  
$V$,
i.e., $V = \coprod_{n=0}^\infty V_{ n }$, and $\dim V_{ 0 }=1$.
  \item [iv.] $V$ is acted on by a group $G$ preserving the above
structure; in particular the form on $V$ is $G$-invariant.
\end{itemize}
Let $ P_{ 1 } = \{ u \in V \otimes V_{\Pi_{1,1}} | L(0 )u
=u, L(i)u =0, i>0\}$. The group $G$ acts on $V \otimes V_{\Pi_{1,1}}$  
via the
trivial action on $V_{\Pi_{1,1}}$. Let ${ P}_{ 1 }^r$ denote the
subspace of $  P_{ 1 }$ of degree $r \in
\Pi_{1,1}$. Then the quotient of $  P_{ 1 }^r$ by the nullspace of its
bilinear form is isomorphic as a $G$-module with $G$-invariant
bilinear form to $V_{ 1- \la r,r\ra /2 }$ if $r \neq 0$ and to $V_{ 1 }
\oplus {\mathbb C^2}$ if $r =0$.
\end{theorem}

Applying Theorem 2 to $V=  V^\natural\otimes V_{  \Pi_{1,1} }$
we see that the monster Lie algebra has $(m,n) \in \mathbb Z \times  
\mathbb Z$ homogeneous
subspaces isomorphic to  the weight spaces $V^\natural_{  mn +1 }$  
when $(m,n)\neq (0,0)$, that is,
$\frak m_{(m,n)}= V^\natural_{[mn]}$. We have shown:
$$
\frak n^+ = \coprod_{m >0 ,n\geq -1 } \frak m_{(m,n)},
$$with
$$ \frak m_{(m,n)} \simeq V^\natural_{mn + 1}
$$
and similarly for $\frak n^-$.

\section{ The structure of the monster Lie algebra}
  A crucial step in \cite{B3} is to identify $\frak m = P_{ 1 }/N 
(\cdot, \cdot)$ with a Lie algebra  given by a generalization of a  
Cartan matrix. This allows one to be able to compute the homology  
groups of the trivial module $\mathbb C$ with respect to an  
appropriate subalgebra of $\mathfrak m$, as can be done for
symmetrizable Kac-Moody Lie algebras \cite{GL}, \cite{Liu}.

  The Lie algebra $\frak g(A)$ associated to a symmetrizable matrix   
$A$ is introduced in \cite{K} and \cite{M}, but the systematic study of the case  
where $A$  satisfies conditions {\bf B1-B3} below was carried out by  
Borcherds.
  Borcherds algebras have many properties in common with  
symmetrizable Kac-Moody algebras such as an invariant bilinear form  
and a root lattice grading. One notable difference is that there may  
be simple imaginary roots in the root lattice. This is a desirable  
property in the monster case $\frak m$ because we wish to associate a  
root grading to the
hyperbolic $\mathbb Z \times \mathbb Z$-grading inherited from $V_ 
{\Pi_{1,1}}$.

We review the construction of the Borcherds algebra $\mathfrak g(A)$  
of \cite{B1}. Let $I$ be a (finite or) countable index set and let $A  
= (a_{ij})_{i,j \in
I}$ be a
matrix with entries in ${\mathbb C}$, satisfying the following  
conditions:
\begin{description}
\item[(B1)]  $A$ is symmetric.
\item[(B2)]  If $ i\neq j$ ($i,j \in I$), then $a_{ij}~\leq~0 $.
\item[(B3)]  If $a_{ii} > 0$ ($i \in I$), then ${2a_{ij} / a_{ii}}
\in \mathbb Z $ for all $j  \in I$.
\end{description}
Let $\frak g'(A)$ be the Lie algebra with generators
$h_{i}, e_i, f_i$, $i \in I$, and the following defining
relations: For all $i,j,k \in I$,

$$ \left[ h_{i}, h_{j}\right] =0, \left[e_i , f_j \right] -   \delta_ 
{ij} h_{i} =0,$$
$$\left[ h_{i}, e_k\right]   -   a_{ik} e_k =0, \left[h_{i}, f_k 
\right]  +   a_{ik} f_k =0$$
and Serre relations
$$ (\ad e_i)^{{-2a_{ij} / a_{ii}} + 1}e_j =0, (\ad f_i)^{{-2a_{ij} /  
a_{ii}} + 1}f_j =0$$
for all $ i \neq j \mbox{ with } a_{ii} > 0$, and finally
$$[e_i, e_j]=0,  [f_i,f_j]=0$$ whenever $a_{ij} =0$.

  Let $\mathfrak h = \sum_{i \in I} \mathbb C h_i$, $\mathfrak n^{\pm} 
$ the subalgebra generated by the elements $e_i$ (resp. the $f_i$)  
for $i \in I=\langle e_i\rangle$. As in the Kac-Moody case, the  
simple roots $\alpha_i \in (\mathfrak h)^*$ are defined to satisfy  $ 
(\alpha_i, \alpha_j) = a_{ij}$. Also as in the Kac-Moody case, we may  
have linearly dependent simple roots $\alpha_i$ and we extend the Lie  
algebra as in \cite{GL} and \cite{L} by an appropriate abelian Lie  
algebra $\mathfrak d$ of degree derivations, chosen so that the  
simple roots are linearly independent in $(\mathfrak h\ltimes  
\mathfrak d)^*$.

  \begin{definition}
The Lie algebra $\frak g(A)= g'(A)\ltimes \mathfrak d$ is the {\it  
Borcherds} or
{\it generalized Kac-Moody} (Lie) algebra associated to the matrix $A$.
Any
Lie algebra of the form $\frak g(A) / \frak c$ where $\frak c$ is a
central ideal is also called a Borcherds algebra.
\end{definition}

Versions of the following theorem appear in \cite{B1} and \cite{B4},  
see also \cite{Jur2}. This theorem
allows us to recognize a Lie algebra associated to a matrix $A$  
satisfying $\bf B1-B3$.
\begin{theorem} Let $\frak g$ be a Lie algebra satisfying the
following conditions:
\begin{enumerate}
  \item  $\frak g$ can be $\mathbb Z$-graded as
  $\coprod_{i\in {\mathbb Z}} \frak g_i$, $\frak g_i$ is finite  
dimensional if $i \neq 0$, and
$\frak g$ is diagonalizable with respect to $\frak g_0$.
\item  $\frak g$ has an involution $\eta$ which maps $\frak g_i$
onto $\frak g_{-i}$ and acts as $-1$ on noncentral elements of
$\frak g_0$, in particular, $\frak g_0$ is abelian.
\item  $\frak g$ has a Lie algebra-invariant bilinear form $(\cdot, 
\cdot)$,
invariant under $\eta$, such that $\frak g_i$ and $\frak g_j$ are
orthogonal if $i \neq -j$, and such that the form $(\cdot,\cdot)_0$,
defined by $(x,y)_0 = -(x, \eta (y))$ for $x,y \in \frak g$, is positive
definite on $\frak g_m $ if $m \neq 0$.
\item  $\frak g_0 \subset [\frak g,\frak g]$.
\end{enumerate}
Then there is a central extension $\hat {\frak g}$ of a Borcherds  
algebra and a homomorphism, $\pi$, from $\hat
{\frak g}$ onto
$\frak g$, such that the kernel of $\pi$ is in the center of $\hat
{\frak g}$.
\label{recognize}
\end{theorem}

The theorem is proven by inductively constructing a set of   
generators of $\mathfrak g_n$, $n \in \mathbb Z$, consisting of $ 
\mathfrak g_0$ weight vectors, using the form  $(x,y)_0 $.  Proofs  
can be found in \cite{B2}, see \cite{Jur2} for the theorem stated  
exactly as above. An alternative characterization of Borcherds  
algebra can be found in \cite{B3}.

\begin{theorem}
The Lie algebra $\frak m
=P_{ 1 }/\mbox{N}(\cdot,\cdot) $ is a Borcherds algebra. \end{theorem}
\begin{proof}
The abelian subalgebra $\frak m_{(0,0)}$ is spanned by elements of the
form $ 1 \otimes  \alpha(-1)\iota(1)$ where $\alpha \in {\Pi_{1,1}}
\otimes_{\mathbb Z}\mathbb C$. Note that $\frak m_{(0,0)}$ is
two-dimensional. In order to apply Theorem \ref{recognize}, grade $ 
\mathfrak \sum_{(m,n) \in \Pi_{1,1}} m_{(m,n)}$ by $i = 2m + n \in  
\mathbb Z$. With this grading,  $\mathfrak m$ satisfies condition (i) of
Theorem \ref{recognize}.

   There is an involution $\eta$ is  on the vertex algebra $V_{\Pi_ 
{1,1}}$, determined by $\eta (\alpha) = -\alpha$ for $\alpha \in \Pi_ 
{1,1}$. Extend the involution to $V^\natural \otimes V_{\Pi_{1,1}}$  
by taking $\eta (\sum u \otimes v) = \sum (u \otimes \eta v)$ for $u  
\in V^\natural, v \in V_{\Pi_{1,1}}$. The invariant form  given by  
Lemma \ref{3.1}  is the required non-degenerate invariant bilinear  
form, satisfying  condition (iii) in Theorem \ref{recognize}. Let $a =  
e^{(1,1)}$,
$b =e^{(1,-1)}$. Condition (iv) follows from the fact that $\mathfrak  
m_{(0,0)}$ is two-dimensional and that the elements
$
[u \otimes\iota(a), v \otimes \iota(a^{-1})] \label{eq:g}
$
  and
$
[\iota(b),\iota(b^{-1})] \label{eq:be}
$
  for  $u,v \in V_{ 2 }^\natural$ are two linearly independent  
vectors in
$\frak m_{(0,0)}$.  Thus the Lie algebra $\mathfrak m$ is the  
homomorphic image of some Borcherds algebra $\mathfrak g(A)$  
associated to a matrix.
\end{proof}

By computing the action of $a \in \mathfrak g_0 = \mathfrak m_{(0,0)} 
$ on $v\in \mathfrak m_{r}$, $r \in \Pi_{1,1}$ one obtains $ [a, v]=  
\langle \alpha, r\rangle v$  and Borcherds identifies the elements
of $\Pi_{1,1}$ with the root lattice of $\frak m$. The following is  
Theorem 7.2 of \cite{B3}.

\begin{theorem}  The simple roots of the monster Lie algebra $ 
\mathfrak m$ are the vectors $(1,n)$, $n = -1$ or $n>0$, each with  
multiplicity $c(n)$.
\end{theorem}
This theorem is  proven \cite{B3} by identifying the product formula
$$
p(J(p)-J(q))=  \prod_{ { i= 1,2, \ldots},  { j = -1,1,\ldots}}
(1-p^iq^j)^{c(ij)}
$$
with the denominator identity for the Borcherds algebra $\mathfrak m$.

  Since, by definition of $\mathfrak m$ the radical of the invariant  
form on $\mathfrak m$ is zero, the kernel of the homomorphism in  
Theorem \ref{recognize} is in the center of $\mathfrak g(A)$. We  
construct a symmetric matrix $B$, determined by the root lattice $\Pi_ 
{1,1}$ and the multiplicities given by Theorem 2.  We have  the Lie  
algebra $\frak m$ is isomorphic to $\frak g(B)/ \frak c$, where
$\frak g(B)$ is the Borcherds algebra associated to the following  
matrix $B$ and $\frak c$ is the full center of $\frak g(B)$:

$${B} = \left( \begin{array}{c|c|c|c}
           2  & \begin{array}{ccc}
           0&\cdots&0    \end{array}
             & \begin{array}{ccc}   -1 & \cdots & -1  \end{array} &  
\cdots \\
          \hline
          \begin{array}{c}  0 \\ \vdots \\ \ 0  \end{array}
                       & \begin{array}{ccc} -2 & \cdots &-2 \\
                                          \vdots & \ddots & \vdots \\
                                        -2 & \cdots & -2   \end{array} &
                       \begin{array}{ccc} -3 & \cdots &-3 \\
                                          \vdots & \ddots & \vdots \\
                                        -3 & \cdots & -3
                                         \end{array} & \cdots\\
    \hline
           \begin{array}{c}  -1 \\ \vdots \\ -1  \end{array} &
           \begin{array}{ccc} -3 & \cdots &-3 \\
                                          \vdots & \ddots & \vdots \\
                                        -3 & \cdots & -3
                                         \end{array} &
           \begin{array}{ccc} -4 & \cdots &-4 \\
                                          \vdots & \ddots & \vdots \\
                                        -4 & \cdots & -4
                                         \end{array} & \cdots \\
           \hline
           \vdots & \vdots & \vdots
          \end{array}   \right)  .   $$

  In this summary of Borcherds' proof of part of the moonshine  
conjectures we are able to bypass the part of the argument of \cite 
{B3}  that requires a more extensive development of the theory of  
Borcherds algebras including generalizing the results of \cite{GL}.  
Instead, we will  use Theorem \ref{thm:free} below, proven in \cite 
{Jur2}. Given a vector space $U$, let $L(U)$ denote the free Lie  
algebra generated by a basis of $U$.
  Let $J \subset I$ be the set $\{ i \in I  | a_{ii} >0\}$. Note that  
the
matrix $(a_{ij})_{i,j \in J}$ is a generalized Cartan matrix. Let
$\frak g_J$ be the Kac-Moody algebra associated to this matrix.
Then $\frak g_J = \frak n^+_J \oplus \frak h_J \oplus \frak n_J^-$, and
$\frak g_J$ is isomorphic to the subalgebra of $\frak g(A)$ generated
by $\{e_i, f_i\}$ with $i \in J$.

\begin{theorem}
  Let $A$ be a matrix satisfying conditions {\bf B1-B3}. Let $J$ and
$\frak g_J$ be as above. Assume
that if $i,j \in I \backslash J$ and $i \neq j$ then $a_{ij}<0$.
Then $$\frak g(A) = \frak u^+ \oplus (\frak g_J + \frak h) \oplus  
\frak u^-,$$
where
$\frak u^- = L(\coprod_{j \in I\backslash J}{  U}(\frak n^-_J) 
\cdot f_j) $
and
$\frak u^+= L(\coprod_{j \in I\backslash J}{  U}(\frak n^+_J)\cdot  
e_j ) $.
The ${  U}(\frak n^-_J)\cdot f_j $ for $j \in I\backslash J$ are  
integrable
highest weight $\frak g_J$-modules, and the ${  U}(\frak n^+_J) 
\cdot e_j $
are integrable lowest weight $\frak g_J$-modules.
\label{thm:free}
\end{theorem}

Note that the conditions on the $a_{ij}$ given in the theorem are
equivalent to the statement that the Lie algebra has no mutually
orthogonal imaginary simple roots. This is the case for the monster  
Lie algebra $\mathfrak m$.

The structure of $\mathfrak m$ can now be summarized.
There are natural isomorphisms
$$\frak m_{(m,n)} \cong V^\natural_{ mn+1 } \mbox{ as an $M$-module  
for }
(m,n) \neq (0,0),$$
$$\frak m_{(0,0)} \cong \mathbb C \oplus \mathbb C, \mbox{ a trivial
$M$-module}.$$
It follows from the definition of $\frak m$ that
$$\frak m_{(-1,1)} \oplus \frak m_{(0,0)}\oplus  \frak m_{(1,-1)} \cong
\frak g \frak l_2.$$
Applying Theorem \ref{thm:free} to the above realization of $ 
\mathfrak m$ by generators and relations gives
$$
\frak m = \frak u^+ \oplus \frak g \frak l_2 \oplus \frak u^- ,
$$
with $ \frak u^- = L({U})$ and $\frak u^+ = L({U}')$.
Where $L({U}),  L({U}')$ are free Lie algebras over vector spaces
that are direct sums of $\frak g\frak l_2$-modules.
$${U} = \coprod_{i >0}W_i \otimes V^\natural_{i +1}\mbox{ and }
{U}' = \coprod_{i >0}{W'_i \otimes V^\natural_{i+1}}.$$
For $i >0$,  $V^\natural_{i+1}$ is (as usual) the weight $i+1$  
component of $V^\natural$,  $W_i$ denotes the (unique up to  
isomorphism) irreducible highest weight $\frak g  \frak l_2$-module of
dimension $i$ on which $z$ acts as $i +1$ and $W'_i$, $i >0$,   
denotes the irreducible lowest weight module.

\section{The homology computation and recursion formulas}
We are now ready to establish the recursion relations for the  
coefficients of the McKay-Thompson series
$\sum_{i >0} {\rm Tr}(g|V_{i +1}^{\natural})q^i = \sum_{i \in \mathbb  
Z}  c_g(i) q^i$.  What follows is a summary of what has
appeared in \cite{jlw}.  See \cite{MR1256622} and \cite{MR1363975}  
for similar computations.

To compute the homology of the free Lie algebra $L(U)$ for a vector
space $U$, with coefficients in the trivial module (as in \cite 
{CE}),  consider
the following exact sequence is a ${  U}(L(U)) = T(U)$-free  
resolution
of the trivial module:
$$
0 \rightarrow T(U) \otimes U {\overset {\mu}  {\rightarrow}} T(U)
{\overset {\epsilon} {\rightarrow}} {\mathbb C}  \rightarrow 0
$$
where $\mu$ is the multiplication map and $\epsilon$ is the
augmentation map.
One obtains :
$$
H_0(L(U),{\mathbb C}) = \mathbb C \label{eq:h0}
$$
$$
H_1(L(U),{\mathbb C}) =  U \cong L(U)/ [L(U), L(U)]
$$
$$
H_n( L(U) ,{\mathbb C}) =0 \mbox{ for } n \geq 2.
$$

Let $p$, $q$ and $t$ be commuting formal variables. The
variables
$p^{-1}$ and $q^{-1}$ will be used to
index the $\mathbb Z \oplus \mathbb Z$-grading of our vector spaces.
All of the $\mathbb M$-modules we encounter are finite-dimensionally
$\mathbb Z \oplus \mathbb Z$-graded with grading suitably truncated  
and will
be identified with formal series in $R(\mathbb M)[[p,q]]$.  
Definitions and
results from \cite{Knu} about the
$\lambda$-ring $R(\mathbb M)$ of finite-dimensional representations  
of $\mathbb M$ are
applicable to formal series in $R(\mathbb M)[[p,q]]$. We summarize  
the results
of, for example, \cite{Knu} that we use below.

The
representation ring $R(\mathbb M)$ is a $\lambda$-ring \cite{Knu}  
with the
$\lambda$ operation given by exterior powers, so $\lambda^i V =
\bigwedge^i V$ for $V \in R(\mathbb M)$.

In the following discussion we let $W,V \in R(\mathbb M)$. The operation
$\bigwedge^i$ satisfies
$$
{\textstyle\bigwedge^i} (W\oplus V)
=\sum_{n=0}^i {\textstyle\bigwedge^n(W)\otimes \bigwedge^{i-n}}(V) .
\label{eq:rul}
$$
Define
$$\textstyle
\bigwedge_t(W) = \bigwedge^0 (W) +\bigwedge^1(W)t +\bigwedge^2 (W)t^2  
+\cdots.
\label{eq:wedge}
$$
Then
\begin{equation}\textstyle
\bigwedge_t (V \oplus W) = \bigwedge_t(V) \cdot \bigwedge_t(W).
\label{eq:times}
\end{equation}
The Adams operations
  $\Psi^k : R(\mathbb M) \rightarrow R(\mathbb M)$ are defined for
$W \in R(\mathbb M)$ by:
\begin{equation}
{\textstyle {\frac {d} { dt}} \log \bigwedge_t(W)} = \sum_{n \geq 0}  
(-1)^n
\Psi^{n+1} (W) t^n .
\label{eq:log}
\end{equation}
  For a class function $f : \mathbb M \rightarrow \mathbb C$, define
$$
(\Psi^k f)(g) = f(g^k).
$$
for all $g \in \mathbb M$.

Now let $W$ be a finite-dimensionally $\mathbb Z \oplus \mathbb Z$- 
graded
representation of $\mathbb M$ such that $W_{(\gamma_1, \gamma_2)}=0$ for
$\gamma_1, \gamma_2 >0$. We shall write
$
W =\sum_{(\gamma_1,\gamma_2) \in \mathbb N^2}W_{(-\gamma_1,
-\gamma_2) }p^{\gamma_1}q^{\gamma_2},\label{eq:fgra}
$
identifying the graded space and formal series.
We extend the definition of $\Psi^k$ to formal series $W\in R(\mathbb  
M)[[p,q]]$
by defining $\Psi^k (p)= p^k$, $\Psi^k (q)= q^k$ and in general,
$$
\Psi^k ( \sum_{(\gamma_1,\gamma_2) \in \mathbb N^2}W_{(-\gamma_1,
-\gamma_2)}p^{\gamma_1}q^{\gamma_2}) =
         \sum_{(\gamma_1,\gamma_2) \in \mathbb N^2}\Psi^k
(W_{(-\gamma_1, -\gamma_2)}) p^{k\gamma_1}q^{k\gamma_2}.
\label{eq:def}$$

   Recall the structure of $\frak u^-$ and ${U} =
H_1(\frak u^-)$ as
$\mathbb Z \oplus \mathbb Z$-graded $M$-modules. We index the
grading by
$p^{-1}$ and $q^{-1}$; then write $\frak u^-$
and ${U} = H_1(\frak u^-)$ as elements of $R[M][[p,q]]$:
\begin{equation}
\frak u^- = \sum_{(m,n)} V^\natural_{ mn+1 } p^m q^n \label{eq:u}
\end{equation}
and
\begin{equation}
U = \sum_{(m,n)} V^\natural_{ m+n } p^m q^n, \label{eq:v}
\end{equation}
where here and below the sums are over all pairs $(m,n)$ such that  
$m,n >0$.

Define
$$H_t(\frak u^-) = \sum_{i=0}^\infty H_i (\frak u^-) t^i$$
and let $H(\frak u^-)$ denote the alternating sum $H_t(\frak
u^-)|_{t={-1}}$.
Recall the Euler-Poincar{\'e} identity:
\begin{equation}\textstyle
\bigwedge_{-1}(\frak u^-) = H(\frak u^-). \label{eq:hom}
\end{equation}
Taking $\log$ of both sides of (\ref{eq:hom}) results in the formal  
power
series identity in $R(M)[[p,q]]\otimes \mathbb Q$:
\begin{equation}
\textstyle
  \log \bigwedge_{-1}(\frak u^-) =  \log H (\frak u^-),\label{eq:lg}
\end{equation}
where we have
$$\log H (\frak u^-)= \log(1 - H_1(\frak u^-)) = -\sum_{n =1}^\infty
{\frac {1} {n}}H_1(\frak u^-)^n. $$
Formally integrating (\ref{eq:log}), with $W = \frak u^-$, gives
$$
{\textstyle \log \bigwedge_{t}}(\frak u^-)=
-\sum_{n \geq 0}  \Psi^{n+1} (\frak u^-) {\frac {(-t)^{n+1}}{ n+1}}.
$$
Then setting $t =-1$ gives:
$$ -\log {\textstyle\bigwedge_{-1}}(\frak u^-) =  \sum_{k=1}^\infty
{\frac {1} { k}}\Psi^k(\frak u^-).$$
Since $H_1(\frak u^-) = {U}$, equation (\ref{eq:lg}) gives
\begin{equation}
\sum_{k=1}^\infty {\frac {1} {k}}\Psi^k(\sum_{(m,n)} V^\natural_{ mn 
+1 }p^m
q^n)
= \sum_{k=1}^\infty {\frac {1} {k}}( \sum_{(m,n)} V^\natural_{ m+n }p^m
q^n)^k
\label{eq:adams}
\end{equation}

  We say $k|(i,j)$ if $k(m,n)=(i,j)$ for some $(m,n) \in \mathbb Z 
\oplus \mathbb Z$.
  For $(i,j) \in \mathbb Z_+ \oplus \mathbb Z_+$ we define
$$P(i,j) = \{ a=(a_{rs})_{r,s \in \mathbb Z_+}\ |\
a_{rs}\in \mathbb N,\ \sum_{(r,s)\in \mathbb Z_+\oplus \mathbb Z_+}
a_{rs}(r,s)=(i,j)\}.$$
We will use the notation $|a| = \sum a_{rs}$, $a! = \prod a_{rs}!$.  
Expanding both sides of
equation (\ref{eq:adams})

\begin{equation}\sum_{(i,j)} \sum_{k|(i,j)}{\frac {1}{k}}\Psi^k(
V^\natural_{ij/k^2 +1}) p^{i}q^{j}
=  \sum_{(i,j)} \sum_{a \in P(i,j)}{\frac{(|a|-1)!}
{a!}} \prod_{r,s \in \mathbb Z_+} (V^\natural_{ r+s })^{a_{rs}} p^i q^j.
  \label{identity}\end{equation}

  Then taking the trace of an element $g \in \mathbb M$ on both sides  
of the identity (\ref{identity})
$$\sum_{(i,j)} \sum_{k| (i,j)}{\frac {1}{ k}}\Psi^k(c_g(ij/k^2)) p^{i} 
q^{j}
$$
$$=  \sum_{(i,j)}\sum_{a \in P(i,j)}{\frac {(|a|-1)!}{ a!}} \prod_ 
{r,s \in \mathbb Z_+}
c_g(r+s-1)^{a_{rs}} p^i q^j. $$
Equating the coefficients of $p^iq^j$  and applying M\"obius  
inversion yields the recursion formulas:

\begin{equation}c_g(ij)
= \sum_{\overset {k>0} {k(m,n)= (i,j)}}{ \frac {1}{ k}} \mu(k)\big 
(\sum_{a
\in P(m,n)} {\frac{(|a|-1)!}{ a!}} \prod_{r,s \in \mathbb Z_+}
c_{g^k}(r+s-1)^{a_{rs}}\big). \label{recursion}\end{equation}

The coefficients of any replicable function are determined by the  
first 23 coefficients \cite{CmN}, but in the case of the above
McKay-Thompson series we can use a smaller set of coefficients.

An examination of the formula (\ref{recursion}) shows that $c_g(n)$ is
determined by expressions of lower level except when $n =1,2,3,5$.
Thus the values of the $c_g(n)$ are determined by the $c_h(1)$,
$c_h(2)$, $c_h(3)$, $c_h(5)$, $h \in \mathbb M$, and the above  
recursions.

As in \cite{B3}, we conclude that
since both the McKay-Thompson series for $V^\natural$ and the
modular functions of \cite{CN} satisfy (\ref{recursion}),
all that is necessary to prove that these functions
are the same is to check the initial data listed above. For the  
modular functions, see \cite{Ko} and  \cite{MR1371745}.
For the relevant initial data about the graded traces of the actions  
of the elements of  the Monster on $V^\natural$, the main theorem of  
\cite{FLM} is needed, as used in \cite{B3}.

  \def\cprime{$'$} \def\cprime{$'$} \def\cprime{$'$}
\providecommand{\bysame}{\leavevmode\hbox to3em{\hrulefill}\thinspace}
\providecommand{\MR}{\relax\ifhmode\unskip\space\fi MR }
\providecommand{\MRhref}[2]{%
   \href{http://www.ams.org/mathscinet-getitem?mr=#1}{#2}
}
\providecommand{\href}[2]{#2}


\end{document}